\documentclass[12pt]{article}

\usepackage{amsmath,amssymb,amsthm}
\usepackage{mathtools}

\usepackage{indentfirst}

\usepackage{graphicx}
\usepackage{tikz}
\usetikzlibrary{arrows.meta,fit}
\usepackage[dvipsnames,svgnames,table]{xcolor}

\usepackage{enumitem}

\usepackage{hyperref}
\hypersetup{
    colorlinks,
    bookmarksnumbered=true
}

\newtheorem{theorem}{Theorem}
\newtheorem{lemma}[theorem]{Lemma}

\theoremstyle{definition}

\newtheorem{problem}[theorem]{Problem}

\theoremstyle{remark}

\newcommand{\ZZ}{\mathbb{Z}}

\newcommand{\abs}[1]{\left\lvert #1 \right\rvert}

\title{Improved Bounds on the Szeged--Wiener Gap and the BKLPS Conjecture}

\author{Lily Zhang\thanks{University of California, Berkeley\@.  Email address: {\tt pink@berkeley.edu}.},
\quad
Evan Li\thanks{University of California, San Diego\@. Email address: {\tt evl012@ucsd.edu}.}}
\date{September 17, 2026}

\begin{document}

\maketitle

\begin{abstract}
    \sloppy
    Bonamy--Knor--Lužar--Pinlou--Škrekovski (2017) define $K_n^t$ to be the complete graph of $n-1$ vertices but with an extra vertex that's adjacent to $t$ vertices of the complete graph part. They propose a stronger conjecture which asserts that if $G$ is a finite simple $2$-connected graph of order $n \ge 10$ not isomorphic to $K_n$, $K_n^2$, nor $K_n^{n-2}$, then the Szeged--Wiener gap of $G$ is $\eta(G) \ge 2n$. We improve upon their work to tighten the bounds on the Szeged--Wiener gap, allowing us to prove this conjecture in the affirmative. Afterwards, we construct graphs attaining equality for each $n \ge 10$ and pose a problem for interested readers to determine a necessary and sufficient condition for equality.
\end{abstract}

\section{Introduction}

Topological indices are often used to study molecular graphs to predict the properties of the corresponding molecules. In particular, one of the most studied ones is the Wiener index, which can be used to help predict boiling points of paraffins~\cite{wiener1947structural}. The original construction underlying the Szeged index was introduced by Gutman~\cite{gutman1994formula}, and it was subsequently developed as the Szeged index by Khadikar et al.~\cite{khadikar1995szeged}. Since the Szeged index extends the edge formula for the Wiener index of trees, the difference between these two indices, the Szeged--Wiener gap, is a natural quantity to study.

The work of Bonamy et al.~\cite{bonamy2017difference} on the Szeged--Wiener gap establishes as a theorem a lower bound on the Szeged--Wiener gap to imply a conjecture of Nadjafi-Arani et al.~\cite{nadjafiarani2012graphs}. They apply their method to strengthen existing work on the Szeged--Wiener gap of bipartite graphs, graphs of girth at least five, and the revised Szeged--Wiener gap.

In our continuation of their work, we adopt their notation, which we reintroduce here with some deviations. Throughout this work, graphs are always deemed finite, simple, and connected.

For $u,v\in V(G)$, denote by $d_G(u,v)$ to be the \textbf{distance} between $u$ and $v$, i.e., the length of the shortest path between $u$ and $v$. The \textbf{Wiener index} is defined to be the sum of distances over all unordered pairs of vertices $\{a,b\} \subset V(G)$
    \[ W(G) = \sum_{\{u,v\} \subseteq V(G)} d_G(u,v). \]
Given an edge $ab \in E(G)$, define
    \[ n_{ab}(a) = \abs{\{ c \in V(G): d_G(c,a) < d_G(c,b) \}}, \]
to be the number of vertices strictly closer to $a$ than to $b$, and define $n_{ab}(b)$ analogously, so $n_{ab}(b) = n_{ba}(b)$ holds. The \textbf{Szeged index} of $G$ is defined to be
    \[ Sz(G) = \sum_{ab \in E(G)} n_{ab}(a) \cdot n_{ab}(b). \]

The quantity that we focus on in this paper is the \textbf{Szeged--Wiener gap} of $G$, which is
    \[ \eta(G)=Sz(G)-W(G). \]
Klavžar, Rajapakse, and Gutman~\cite{klavzar1996szeged} proved that $Sz(G) \ge W(G)$ for every connected graph $G$, so this gap is nonnegative.

Let $n,t$ be integers with $t \in [1, n-1]$. The notation $K_n^t$ is used for the graph obtained from the complete graph on $n-1$ vertices $K_{n-1}$ by adding one vertex adjacent to exactly $t$ vertices of $K_{n-1}$. We also denote by $C_n$ the cycle graph on $n$ vertices.

Nadjafi-Arani, Khodashenas, and Ashrafi~\cite{nadjafiarani2012graphs} classified graphs with small Szeged--Wiener gap and proposed a lower bound in terms of the noncomplete blocks of a graph. Klavžar and Nadjafi-Arani~\cite{klavzar2014improved} later established stronger bounds dependent on the girth and the length of a longest isometric cycle. Bonamy et al.~\cite{bonamy2017difference} proved that every $2$-connected noncomplete graph $G$ of order $n$ has Szeged--Wiener gap $\eta(G) \ge 2n-6$, characterized the equality cases, and then proposed a stronger conjecture involving $K_n^t$.

The main goal of this work is to prove the stronger conjecture of Bonamy et al.~\cite[Conjecture~5]{bonamy2017difference}, which states that when $G$ is $2$-connected of order $n \ge 10$ satisfying $G \not\cong K_n, K_n^2, K_n^{n-2}$, then
    \[ \eta(G) \ge 2n. \]
We refer to this conjecture as the BKLPS Szeged--Wiener Gap Conjecture, and this is later proved as Theorem~\ref{theorem:6}.

Due to their prominence in this work, we say the graph $G$ is an $n$-\textbf{exceptional} graph when it's isomorphic to the particular graphs $K_n$, $K_n^2$, or $K_n^{n-2}$, and $n$-\textbf{unexceptional} otherwise. We will rephrase the conjecture in terms of this terminology when it comes time to prove it in Theorem~\ref{theorem:6}.

To do this, we expand upon the results of Bonamy et al.~\cite{bonamy2017difference} by proving theorems strengthening their work. Afterwards, equality cases for the bounds we show in the paper are also discussed. We present a problem for interested readers to continue on with this research on the Szeged--Wiener gap.
\section{Preliminaries}

We first introduce the terminology of Bonamy et al.~\cite{bonamy2017difference} that will be used throughout the proofs.

Given distinct $a,b \in V(G)$, an edge $xy \in E(G)$ is considered $\{a,b\}$-\textbf{good} if
\begin{align*}
    d_G(a,x) < d_G(a,y)&, \; d_G(b,y) < d_G(b,x) \\
        &\text{or} \\
    d_G(a,y) < d_G(a,x)&, \; d_G(b,x) < d_G(b,y)
\end{align*}
Let $g_G(a,b)$ be the number of $\{a,b\}$-good edges of $G$, with
    \[ \eta_G(a,b)=g_G(a,b)-d_G(a,b). \]
By convention, we define $g_G(a,a), \eta_G(a,a) = 0$.

Given $u\in V(G)$, let $N_G(u)$ and $N_G[u]$ denote the open and closed
neighborhoods of $u$ respectively, so $N_G[u] = N_G(u) \cup \{u\}$.

A \textbf{block} of $G$ is a subsgraph $B \subseteq G$ which is maximal with no cut-vertex.

The \textbf{contribution} of a vertex $a \in V(G)$ is defined to be
    \[ c_G(a) = \sum_{b \in V(G)} \eta_G(a,b) = \sum_{b \in V(G)} [ g_G(a,b) - d_G(a,b) ] \]
Rather than invoke Bonamy et al.~\cite[Lemma~11]{bonamy2017difference}, we show a significantly stronger identity demonstrating the exact Szeged--Wiener gap of $K_n^t$.

\begin{theorem}\label{theorem:1}
    Let integers $n,t$ satisfy $t \in [1, n-1]$. Then
        \[ \eta(K_n^t), \; \eta(K_n^{n-t}) = 2(t-1)(n-t-1). \]
\end{theorem}
\begin{proof}
    We only need to show $\eta(K_n^t) = 2(t-1)(n-t-1)$, then it would immediately follow that
        \[ \eta(K_n^{n-t}) = 2(n-t-1)(t-1) = 2(t-1)(n-t-1). \]
    Let $G = K_n^t$, and let $z \in V(K_n^t)$ be the vertex with degree $t$, defining the partitions of the remaining $n-1$ vertices
        \[ A = N_G(z), \; B = V(G) \setminus (A \cup \{z\}). \]
    Note that $\abs{A} = t$, $\abs{B} = n-t-1$.

    The $\binom {n-1}2$ pairs within $A \cup B$ have distance $1$, each $a \in A$ satisfies $d_G(z, a) = 1$, and each $b \in B$ satisfies $d_G(z,b) = 2$. Therefore the Wiener index of $G$ is
        \[ W(G) = \binom {n-1}2 + t + 2(n-t-1). \]
    
    We now compute the Szeged index by considering edges in $G$.
    \begin{itemize}
        \item Suppose first the pair $\{x,y\} \subseteq A$ or $B$ (with $\binom t2 + \binom{n-t-1}2$ possibilities). $z$ has equal distance from $x$ and $y$. All other vertices $a \in V(G) \setminus \{x,y,z\}$ satisfy $d_G(a, x) = 1 = d_G(a, y)$, so
            \[ n_{xy}(x), \; n_{xy}(y) = 1. \]
        \item Suppose $a \in A$, $b \in B$ (with $t(n-t-1)$ possibilities). Only $z$ is closer to $a$ than $b$ with the remaining $n-3$ vertices $x \not= z,a,b$ satisfy $d_G(x, a) = 1 = d_G(x, b)$. Therefore
            \[ n_{ab}(a) = 2, \; n_{ab}(b) = 1. \]
        \item The final case consists of the edges between $a \in A$ and $z$ (with $t$ possibilities). The only vertices closer to $a$ are the vertices in $B$. On the other hand, the other vertices in $A$ are adjacent to both $a$ and $z$, therefore
            \[ n_{az}(a) = n-t, \; n_{az}(z) = 1. \]
    \end{itemize}
    So
    \begin{align*}
        Sz(G)
        &= \binom t2 + \binom{n-t-1}2 + 2t(n-t-1) + t(n-t) \\
        &= \binom{n-1}2 + t + 2t(n-t-1).
    \end{align*}
    We conclude
        \[ \eta(G) = Sz(G) - W(G) = 2(t-1)(n-t-1), \]
    which was what we wanted.
\end{proof}
\section{Intermediate Bounds}

We start by proving an exceptional deletion lemma which improves upon the result of Bonamy et al.~\cite[Lemma~8]{bonamy2017difference}.

\begin{lemma}[Exceptional Deletion]\label{lemma:2}
    Let $G$ be an $n$-unexceptional $2$-connected graph of order $n \ge 4$, and let distinct $u, v \in V(G)$ satisfy $N_G[u] \subseteq N_G[v]$. If $G-u \cong K_{n-1}^2$ or $K_{n-1}^{n-3}$, then
        \[ \eta(G) - \eta(G-u) \ge n-2. \]
\end{lemma}
\begin{proof}
    Let $H = G-u$. One can verify that when $n=4$, $G$ being $n$-unexceptional, $2$-connected, and $H \cong K_{n-1}^2$ or $K_{n-1}^{n-3}$ is not possible.

    If $n=5$, then $H = K_4^2$ with Szeged--Wiener gap $\eta(H) = 9-7 = 2$. A finite check on $G$ shows that $\eta(G) \ge 5$, so $\eta(G) - \eta(H) \ge 3 = 5-2$.

    From here on out, we assume $n \ge 6$ and use the proof by Bonamy et al.~\cite[Lemma~7]{bonamy2017difference}. Its conditions are satisfied since both $K_{n-1}^2$ and $K_{n-1}^{n-3}$ are $2$-connected and noncomplete. Following the notation of their proof, let $q$ be the number of couples $(\{a,b\}, w)$ such that $a,b \in V(H)$, $w \in N_G(u)$, and $uw$ is $\{a,b\}$-good.
    Bonamy et al.~\cite[Lemma~7, Eq.~(4)]{bonamy2017difference} give the identity
        \[ \eta(G) - \eta(H) = \sum_{x \in V(H)} \eta_G(u,x) + q. \tag{1} \]
    When $x \not\in N_G[u]$, the number of neighbors $w$ of $x$ satisfying $d_G(u,w) < d_G(u,x)$ is denoted
        \[  p_x = \abs{\{ w \in N_G(x): d_G(u,w) < d_G(u,x) \}}. \]
    Bonamy et al.~\cite[Lemma~7, (C2)]{bonamy2017difference} give the estimate
        \[ \eta_G(u,x) \ge 2(p_x - 1). \tag{2} \]
    Finally, Bonamy et al.~\cite[Lemma~7, (C1)]{bonamy2017difference} state that if $t$ is the number of non-edges in $N_G(u)$, then
        \[ q \ge 2t. \tag{3} \]
    Those are the facts we use from their proof, so we may begin. We consider casework on whether $H \cong K_{n-1}^2$ or $K_{n-1}^{n-3}$.

    \textbf{Case 1 ($H \cong K_{n-1}^2$).} Let $z \in V(H)$ be the vertex with degree $2$, writing $N_H(z) = \{a,b\}$. Let $R = V(H) \setminus \{z,a,b\}$, and $S = N_G(u)$, with $r, s$ denoting their respective sizes. We split into subcases depending on whether $z \in S$.

    \textbf{Subcase 1a ($z \not\in S$).} This means $S \subseteq V(H) \setminus \{z\}$, so $s \in [2, n-2]$. The lower bound follows from $2$-connectedness. If $s = n-2$, then $S = V(H) \setminus \{z\}$, further implying $G \cong K_n^2$. This is a contradiction, so $s \in [2, n-3]$.

    There are $n-2-s$ vertices in $H-z \cong K_{n-2}$ but not $S$. Let $x$ be one of them. So $x$ is adjacent to all vertices of $S$ and $d_G(u,x) = 2$, therefore $p_x = s$. By (1) and (2) respectively,
        \[ \eta(G) - \eta(H) \ge \sum_{x \in V(H)} \eta_G(u,x) \ge 2(n-2-s)(s-1). \]
    We know $n-2-s, s-1 \ge 1$, and $(n-2-s) + (s-1) = n-3$ is constant in $s$, so the product is minimized when $s-1$ or $n-2-s = 1$. Therefore
        \[ 2(n-2-s)(s-1) \ge 2 \cdot 1 (n-4) \ge n-4 + 2 = n-2. \]
    
    \textbf{Subcase 1b ($z \in S$).} We first claim $S \cap \{a,b\} \not= \varnothing$. Clearly $v \in S$ since $u \in N_G[u] \subseteq N_G[v]$. Also, $z \in N_G[u] \subseteq N_G[v]$. If $v \not= z$, then $v$ is adjacent to $z$, so $v \in \{a,b\}$ is our desired element. If $v=z$, then $\abs{S} \ge 2$ by $2$-connectivity, so there exists $w \in S \setminus \{z\}$. $N_G[u] \subseteq N_G[z]$ implies $w$ is adjacent to $z$, so $w \in \{a,b\}$. This proves the claim. 
    
    Now let $T = R \cap S$ with $t = \abs{T}$. First, suppose $t \ge 1$. Each $y \in T$ is not adjacent to \(z\), but $y,z \in S$. This easily implies $q \ge 2t$ by (3).

    Let $x \in R \setminus T$. Since $S \cap \{a,b\} \not= \varnothing$, $d_G(u,x) = 2$. Moreover, $x$ is adjacent to every vertex $y \in T$ and to at least one of $a,b$ (the one in $S$), therefore $p_x \ge t+1$. There are $n-4-t$ such vertices $x$, so (1) and (2) gives us
    \begin{align*}
        \eta(G) - \eta(H)
        &\ge \sum_{x \in V(H)} \eta_G(u,x) + q \\
        &\ge (n-4-t) \cdot 2t + 2t \\
        &= 2t(n-3-t).
    \end{align*}
    Now $t$, $n-3-t \ge 1$, and $(t) + (n-3-t) = n-3$ is constant in $t$, therefore their product is minimized when $t$ or $n-3-t = 1$. So
        \[ 2t(n-3-t) \ge 2 \cdot 1(n-4) \ge n-4 + 2 = n-2, \]
    which proves the $t \ge 1$ case.

    Now suppose $t=0$. If $a, b \in S$, then for each of the $n-4$ vertices $x \in R$, $p_x = 2$ (the 2 are $a$ and $b$). Again by (1) and (2)
        \[ \eta(G) - \eta(H) \ge \sum_{x \in V(H)} \eta_G(u,x) \ge 2(n-4) \ge n-2. \]
    The final case is if $S = \{z,a\}$ or $\{z,b\}$, where we assume without loss of generality that $S = \{z,a\}$. Clearly $p_b = 2$ (the 2 are $z$ and $a$).
    
    Finally, let $x \in R$. The path $uax$ $u$--$x$ path, so its two edges are $\{u,x\}$-good. The edge $zb$ is also $\{u,x\}$-good, since $d_G(u,z) = 1 < 2 = d_G(u,b)$ while $d_G(x,b) = 1 < 2 = d_G(x,z)$. We conclude $g_G(u,x) \ge 3$, so $\eta_G(u,x) = g_G(u,x) - d_G(u,x) \ge 3-2 = 1$. Now (1) gives us
        \[ \eta(G) - \eta(H) \ge \sum_{x \in V(H)} \eta_G(u,x) \ge 2 + (n-4) \cdot 1 \ge n-2. \]
    This discloses case 1.

    \textbf{Case 2 ($H \cong K_{n-1}^{n-3}$).} Again, let $S = N_G(u)$ with $s = \abs{S}$, satisfying $s \in [2, n-1]$. If $s = n-1$, then $S = V(H)$, implying $G \cong K_n^{n-2}$. This is a contradiction, so $s \in [2, n-2]$.

    Let $x \in V(H) \setminus S$. The graph $H \cong K_{n-1} - e$ has a missing edge $e$, and $s \ge 2$, so $x$ is adjacent to a vertex in $S$. We derive $d_G(u,x) = 2$, hence
        \[ p_x = \abs{\{ w \in N_G(x): d_G(u,w) < 2 \}} = \abs{N_H(x) \cap S}. \]
    Also, $H \cong K_{n-1} - e$ shows that at most one vertex $x \in V(H)\setminus S$ isn't adjacent to some $y \in S$. At least $n-2-s$ vertices $x \in V(H) \setminus S$ satisfy $p_x = s$, while the possible remaining vertex satisfies $p_x \ge s-1$. Using (1) and (2) for the last time, we get
    \begin{align*}
        \eta(G) - \eta(H)
        &\ge \sum_{x \in V(H)} \eta_G(u,x) \\
        &\ge (n-2-s) \cdot 2(s-1) + 1 \cdot 2(s-2) \\
        &= 2(s-1)(n-1-s) - 2.
    \end{align*}
    Again, $s-1$, $n-1-s \ge 1$, and $(s-1) + (n-1-s) = n-2$ is constant in $s$, so the product is minimized when $s-1$ or $n-1-s = 1$. Finally
        \[ 2(s-1)(n-1-s) - 2 \ge 2 \cdot 1(n-3) - 2 \ge n-3 + 3-2 = n-2, \]
    exactly as desired.
\end{proof}

\begin{theorem}\label{theorem:3}
    Let $G$ be an $n$-unexceptional $2$-connected graph of order $n \ge 6$. Then
        \[ \eta(G) \ge 2n-4. \]
\end{theorem}
\begin{proof}
    We plan to prove the statement by strong induction on $n$, where the base case $n=6$ is a finite check on the $156$ unlabeled graphs showing that $53$ are $6$-unexceptional and $2$-connected. Among them, the minimum Szeged--Wiener gap is $8 = 2 \cdot 6 - 4$.
    
    Now assume the statement holds up to $n-1 \ge 6$. If there are no distinct vertices $u, v \in V(G)$ satisfying $N_G[u] \subseteq N_G[v]$, then Bonamy et al.~\cite[Lemma~10]{bonamy2017difference} imply $\eta(G) \ge 2n$ (since $G \not\cong C_5$ is noncomplete), so we're done. Therefore, assume such a vertex pair $u,v$ exists, and let $H = G-u$.

    First suppose $H$ is $2$-connected and noncomplete, so Bonamy et al.~\cite[Lemma~7]{bonamy2017difference} directly imply $\eta(G) - \eta(H) \ge 2$.\\
    If $H$ is $(n-1)$-unexceptional, then our own induction hypothesis gives us
        \[ \eta(G) \ge 2 + \eta(H) \ge 2 + 2(n-1) - 4 = 2n - 4. \]
    Else if $H$ is $(n-1)$-exceptional, then since $H$ is noncomplete, Lemma~\ref{lemma:2} implies $\eta(G) - \eta(H) \ge n-2$. Combined with $\eta(H) = 2n-8$ by Theorem~\ref{theorem:1} (regardless of whether $H \cong K_{n-1}^2$ or $K_{n-1}^{n-3}$), we have
        \[ \eta(G) \ge n-2 + \eta(H) = n-2 + (2n-8) \ge 2n-4. \]
    We can therefore assume that for every distinct $u,v \in V(G)$ satisfying $N_G[u] \subseteq N_G[v]$, $G-u$ is either complete or not $2$-connected. Fix $u,v$ to be one such pair (and $H = G-u$).
    
    Following the proof by Bonamy et al.~\cite[Theorem~4, proof]{bonamy2017difference}, let $C_1, \dotsc, C_k$ be the blocks of $H$, and let $G_i = G[V(C_i) \cup \{u\}]$ with order $m_i$. Their proof already shows that $k \ge 2$, $G_i$ has order $m_i < n$, and the equations
    \begin{align*}
        \eta(G) &\ge \sum_{i=1}^k \eta(G_i) + 2\binom k2 \tag{1} \\
        \sum_{i=1}^k m_i &= n + 2(k-1) \tag{2}
    \end{align*}
    hold. If $m_i = 5$, then Bonamy et al.~\cite[Theorem~4]{bonamy2017difference} directly give us $\eta(G_i) \ge 2m_i - 5$. Else $m_i \ge 6$, where our induction hypothesis implies $\eta(G_i) \ge 2m_i - 4$.

    Therefore, if $r$ is the number of blocks with $m_i = 5$, then combining (1) and (2) yields
    \begin{align*}
        \eta(G)
        &\ge \sum_{i=1}^k \eta(G_i) + 2\binom k2 \\
        &\ge 2 \sum_{i=1}^k m_i - 5r - 4(k-r) + k(k-1) \\
        &= 2n + 4k - 4 - 5r - 4k + 4r + k(k-1) \\
        &= 2n - 4 - r + k(k-1) \\
        &\ge 2n-4,
    \end{align*}
    where the last inequality follows from $k(k-1) \ge k \ge r$. The induction is finished.
\end{proof}

Now that this intermediate bound is proven, we may start with the proof of the conjecture.
\section{Proof of Conjecture}

We start with 2 lemmas just for the main purpose of the conjecture, then we will prove the conjecture as Theorem~\ref{theorem:6}. These lemmas have very short proofs, but they are clean statements that allow us to push the bounds enough to resolve the conjecture.

\begin{lemma}\label{lemma:4}
    Let $G$ be a $2$-connected graph, and let distinct $u,v \in V(G)$ satisfy $N_G[u] \subseteq N_G[v]$. Suppose that $G-u$ is not $2$-connected, and $v$ is its unique cut-vertex. Let $C$ be a $\abs{V(C)}$-unexceptional block of $G-u$, defining $G_0 = G[V(C) \cup \{u\}]$ with order $m$. Then
        \[ \eta(G_0) \ge 2m-4. \]
\end{lemma}
\begin{proof}
    Suppose $m \le 4$. Then $\abs{V(C)} \le 3$, but the only such $2$-connected graph is $C \cong K_3$, contradicting noncompleteness. So $m \ge 5$.

    If $m \ge 6$, then since $G_0$ is $2$-connected and $m$-unexceptional, Theorem~\ref{theorem:3} immediately implies $\eta(G_0) \ge 2m-4$.

    It remains to consider $m=5$, so $\abs{V(C)} = 4$. Being $2$-connected, the possibilities for $C$ are $C_4$, $K_4^2$, and $K_4$. The latter two are $4$-exceptional, so $C \cong C_4$.

    It clearly follows that $N_{G_0}[u] \subseteq N_{G_0}[v]$. Since $G_0 - u = C \cong C_4$ is $2$-connected and noncomplete, Bonamy et al.~\cite[Lemma~7]{bonamy2017difference} imply
        \[ \eta(G_0) - \eta(C) \ge 2. \]
    The Szeged--Wiener gap of $C$ is $\eta(C) = \eta(C_4) = 16 - 8 = 8$, so  
        \[ \eta(G_0) \ge 2 + \eta(C) = 10 \ge 2 \cdot 5 - 4, \]
    resolving the $m=5$ situation.
\end{proof}

\begin{lemma}\label{lemma:5}
    Let $G$ be a $2$-connected graph, and let distinct $u,v$ satisfy\\
    $N_G[u] \subseteq N_G[v]$. Suppose that $G-u$ is not $2$-connected, and $v$ is its unique cut-vertex. Let $C_1,C_2$ be distinct $2$-connected blocks of $G-u$. Then
        \[ \sum_{\substack{a \in V(C_1) \setminus \{v\} \\ b\in V(C_2) \setminus \{v\}}} \eta_G(a,b) \ge 4. \]
\end{lemma}
\begin{proof}
    By $2$-connectedness of $G$, let $u$ have neighbors $w_1 \in V(C_1) \setminus \{v\}$, $w_2 \in V(C_2) \setminus \{v\}$. From $N_G[u] \subseteq N_G[v]$ and the fact that $w_1, w_2$ are in distinct blocks, we know $\eta_G(w_1, w_2) \ge 2$.

    By $2$-connectedness of $C_i$ and the fact that $w_i \in N_G(u) \subset N_G[v]$, $w_i$ has a neighbor (besides $v$) $x_i \in V(C_i) \setminus \{v,w_i\}$. Our big claim is the following:
        \[ \eta_G(x_i, w_j) \ge 1, \tag{1} \]
    where $i, j \in \{1,2\}$ are distinct.

    If $x_i \in N_G(u)$, then $N_G[u] \subseteq N_G[v]$ again implies $\eta_G(x_i, w_j) \ge 2 \ge 1$. Else if $x_i \not\in N_G(u)$, then it follows that the edge $w_i u$ is $\{x_i,w_j\}$-good. Indeed,
        \[ d_G(x_i,w_i) = 1 < d_G(x_i,u) \]
    and
        \[ d_G(w_j,u) = 1 < 2 \le d_G(w_j,w_i). \]
    We note that $d_G(w_j,w_i) \ge 2$ since they are in different blocks of $G-u$.

    $N_G[u] \subseteq N_G[v]$ also implies there's a shortest $x_i$--$w_j$ path in $G$ that doesn't go through $u$. Each edge of this path is $\{x_i,w_j\}$-good. The edge $w_i u$ is also $\{x_i,w_j\}$-good, hence
        \[ g_G(x_i,w_j) \ge d_G(x_i,w_j) + 1, \]
    which demonstrates (1).

    The three pairs $\{w_1,w_2\}$, $\{x_1,w_2\}$, $\{x_2,w_1\}$ are distinct, therefore the desired sum is
    \begin{align*}
        \sum_{\substack{a \in V(C_1) \setminus \{v\} \\ b\in V(C_2) \setminus \{v\}}} \eta_G(a,b)
        &\ge \eta_G(w_1,w_2) + \eta_G(x_1,w_2) + \eta_G(w_1,x_2) \\
        &\ge 2 + 1 + 1 \\
        &= 4,
    \end{align*}
    as required.
\end{proof}

With these lemmas in store, we are ready to attack the conjecture so that it's no longer a conjecture.

\begin{theorem}\label{theorem:6}
    Let $G$ be an $n$-unexceptional $2$-connected graph of order $n \ge 10$. Then
        \[ \eta(G) \ge 2n. \]
\end{theorem}
\begin{proof}
    We proceed via strong induction on $n$, where the base case $n=10$ was already computationally resolved by Bonamy et al.~\cite[p.~203]{bonamy2017difference}.

    Assume that the statement holds up to $n-1 \ge 10$. The graph $G$ is noncomplete, so if there are no distinct vertices $u,v \in V(G)$ satisfying $N_G[u] \subseteq N_G[v]$, then $\eta(G) \ge 2n$ by Bonamy et al.~\cite[Lemma~10]{bonamy2017difference}. We would be done, so assume such vertices $u,v$ exist.

    Suppose we can pick such a pair $u,v$ such that $H = G-u$ is $2$-connected and noncomplete. First, if $H$ is $(n-1)$-unexceptional, then summing the relations obtained from Bonamy et al.~\cite[Lemma~7]{bonamy2017difference} and the induction hypothesis respectively
    \begin{align*}
        \eta(G) - \eta(H) &\ge 2 \\
        \eta(H) &\ge 2(n-1)
    \end{align*}
    yields
        \[ \eta(G) \ge 2 + 2(n-1) = 2n, \]
    as desired. Otherwise, if $H$ is $(n-1)$-exceptional, then $H \cong K_{n-1}^2$ or $K_{n-1}^{n-3}$. This time, summing the relations from Lemma~\ref{lemma:2} and Theorem~\ref{theorem:1} respectively
    \begin{align*}
        \eta(G) - \eta(H) &\ge n-2 \\
        \eta(H) &= 2(2-1)(n-1-2-1) = 2n-8.
    \end{align*}
    yields
        \[ \eta(G) \ge 3n-10 \ge 2n. \]
    
    Suppose now that for each distinct $u,v \in V(G)$ satisfying $N_G[u] \subseteq N_G[v]$, the graph $G-u$ is either complete or not $2$-connected. Fix $u,v$ to be one such pair (and $H = G-u$).

    Again, we use the notation from the proof by Bonamy et al.~\cite[Theorem~4, proof]{bonamy2017difference} (where we previously used it to prove Theorem~\ref{theorem:3}). Let $C_1, \dotsc, C_k$ be the blocks of $H$, and let $G_i = G[V(C_i) \cup \{u\}]$ with order $m_i$. Their proof already shows that $k \ge 2$ and the relations
    \begin{align*}
        \eta(G) &\ge \sum_{i=1}^k \eta(G_i)
        + \sum_{1 \le i < j \le k} \sum_{\substack{a \in V(C_i) \setminus \{v\} \\ b \in V(C_j) \setminus \{v\}}} \eta_G(a,b) \tag{1} \\
        \sum_{i=1}^k m_i &= n + 2(k-1) \tag{2}
    \end{align*}
    hold. Combining everything we know, we compute
    \begin{align*}
        \eta(G)
        &\ge \sum_{i=1}^k \eta(G_i)
        + \sum_{1 \le i < j \le k} \sum_{\substack{a \in V(C_i) \setminus \{v\} \\ b \in V(C_j) \setminus \{v\}}} \eta_G(a,b) \tag{Relation (1)} \\
        &\ge \sum_{i=1}^k (2m_i - 4)
        + \sum_{1 \le i < j \le k} 4 \tag{Lemmas~\ref{lemma:4} and~\ref{lemma:5}} \\
        &= 2n + 4(k-1) - 4k + 4 \binom k2 \tag{Relation (2)} \\
        &= 2n - 4 + 2k(k-1) \\
        &\ge 2n, \tag{$k \ge 2$}
    \end{align*}
    which finishes the induction step that we hoped to show.
\end{proof}
\section{Further Results}

It in fact follows that the inequality is sharp no matter what constraint $n \ge a$ we place when $a \in \ZZ$ is a finite constant. Therefore, a stronger result would require reworking the statement more than simply changing the constant $a=10$. We invite the reader to determine a necessary and sufficient condition for when equality holds, and we phrase it as a problem here.

\begin{problem}
    Find all integers $n \ge 10$ and $n$-unexceptional $2$-connected graphs of order $n$ such that
        \[ \eta(G) = 2n. \]
\end{problem}

We now show the stronger sharpness result here, which demonstrates a sufficient condition for equality. One remark is that this definitely isn't a necessary condition, because counterexamples can be found.

\begin{lemma}
    Let $n \ge 10$ be an integer. There exists an $n$-unexceptional $2$-connected graph of order $n$ such that
        \[ \eta(G) = 2n. \]
\end{lemma}
\begin{proof}
    Consider $Q \cong K_{n-2}$. Choose distinct $a,b \in V(Q)$, and add $2$ new vertices $x,y$ with additional edges $ax, by, xy$, calling this graph $G = G_n$. It's easy to see that this choice of $G$ is both $2$-connected and $n$-unexceptional. We only need to show that $\eta(G) = 2n$.

    $G_n$ has diameter $2$ with only non-edges being the $2n-6$ pairs of vertices
        \[ \{ xq: q\in V(Q) \setminus \{a\} \} \text{ and } \{ yq: q \in V(Q) \setminus \{b\}\}. \]
    Therefore
        \[ W(G_n) = \binom n2 + (2n-6) = \frac{n^2 + 3n - 12}2. \]
    For the sake of computing the Szeged index, let $A = V(Q) \setminus \{a,b\}$, so $\abs{A} = n-4$.
    \begin{itemize}
        \item If $c,d \in A$ (with $\binom{n-4}2$ possibilities), then clearly
            \[ n_{cd}(c), \; n_{cd}(d) = 1. \]
        \item If $c \in A$, then consider the edge $ac$ or $bc$ (with $2(n-4)$ possibilities). We WLOG only consider $ac$, where we note that only $x$ is closer to $a$ than $c$ (besides itself). No vertex besides $c$ is closer to $c$ than $a$, so
            \[  n_{ac}(a) = 2, \; n_{ac}(c) = 1. \]
        \item Just considering the edge $ab$, only $x$ is closer to $a$, and $y$ is closer to $b$, so
            \[ n_{ab}(a), \; n_{ab}(b) = 2. \]
        \item Consider the edge $ax$ (the $by$ case is similar with a total of $2$ possibilities). Note that $d_G(a,y)$, $d_G(c,x) = 2$. Every vertex not $a$, $x$, nor $y$ is adjacent to $a$, but only $y$ is adjacent to $x$ (that isn't $a$ itself). So
            \[ n_{ax}(a) = n-2, \; n_{ax}(x)= 2. \]
        \item Just considering the edge $xy$, only $a$ is closer to $x$ than $y$ besides itself. A similar result holds for $y$, so
            \[ n_{xy}(x), \; n_{xy}(y) = 2. \]
    \end{itemize}
    Therefore
    \begin{align*}
        Sz(G_n)
        &= \binom{n-4}2 + 2 \cdot 2(n-4) + 4 + 2 \cdot 2(n-2) + 4 \\
        &= \binom{n-4}2 + 8n - 16 \\
        &= \frac{n^2+7n-12}2.
    \end{align*}
    Our conclusion is that
        \[ \eta(G_n) = Sz(G_n) - W(G_n) = \frac{n^2 + 7n - 12 - (n^2 + 3n - 12)}2 = 2n. \]
\end{proof}

\bibliographystyle{unsrt}
\bibliography{references}

\end{document}